\documentclass{amsart}
\usepackage{amssymb}
\usepackage[utf8]{inputenc} 
\usepackage[english]{babel} 

\theoremstyle{plain} 
\newtheorem{thm}{Theorem} 
\newtheorem{lem}[thm]{Lemma} 
\newtheorem{cor}[thm]{Corollary}

\theoremstyle{definition}

\theoremstyle{remark} 
 
\newtheorem{que}{Question} 
\newtheorem{conj}{Conjecture}

\DeclareMathOperator{\mre}{Re}  

\begin{document} 
\title{Contractive inequalities for Hardy spaces} 
\date{\today} 

\author[Brevig]{Ole Fredrik Brevig} \address{Department of Mathematical Sciences, Norwegian University of Science and Technology, NO-7491 Trondheim, Norway} \email{ole.brevig@ntnu.no}

\author[Ortega-Cerd\`{a}]{Joaquim Ortega-Cerd\`{a}} \address{Department de Matem\`{a}tiques i Inform\`{a}tica, Universitat de Barcelona \& Barcelona Graduate school in mathematics, Gran Via 585, 08007 Barcelona, Spain} \email{jortega@ub.edu}

\author[Seip]{Kristian Seip} \address{Department of Mathematical Sciences, Norwegian University of Science and Technology, NO-7491 Trondheim, Norway} \email{kristian.seip@ntnu.no}

\author[Zhao]{Jing Zhao} \address{Department of Mathematical Sciences, Norwegian University of Science and Technology, NO-7491 Trondheim, Norway} \email{jingzh95@me.com}

\thanks{The research of Brevig, Seip, and Zhao was supported by Grant 227768 of the Research Council of Norway. The research of Ortega-Cerd\`{a} was supported by grant  MTM2014-51834-P  of the
Ministerio de Econom\'{\i}a y Competitividad,
and by the Generalitat de Catalunya (project 2014 SGR 289).}


\begin{abstract}
We state and discuss several interrelated results, conjectures, and questions regarding contractive inequalities for classical $H^p$ spaces of the unit disc. We study both coefficient estimates in terms of weighted $\ell^2$ sums and the Riesz projection viewed as a map from $L^q$ to $H^p$ with $q\ge p$. Some numerical evidence is given that supports our conjectures. 
\end{abstract}
\dedicatory{Dedicated to Professor Pawe{\l} Doma\'{n}ski in memoriam}

\maketitle
\section{Introduction} 
This paper deals with certain contractive inequalities for the classical Hardy spaces $H^p$ of the unit disc $\mathbb{D}$, where as usual $f$ belongs to $H^p$ for $0<p<\infty$ if $f$ is analytic in $\mathbb{D}$ and 
\[ \| f\|_{H^p}:=\sup_{0<r<1} \left(\int_{0}^{2\pi} |f(re^{i\theta})|^p \frac{d\theta}{2\pi}\right)^{\frac{1}{p}}<\infty. \] Carleman's inequality (see \cite{Vukotic03} for an excellent exposition), which states that 
\begin{equation} \label{eq:Carl} \left(\sum_{n=0}^\infty \frac{|a_n|^2}{n+1}\right)^\frac{1}{2} \leq \|f\|_{H^1}\end{equation}
for $f(z)=\sum_{n\geq0} a_n z^n$, is a prototypical example of the kind of inequality we are interested in. It is well known that \eqref{eq:Carl} belongs to a family of inequalities that appear in the following way. For $\alpha\geq1$, define  $c_\alpha$ as the coefficient sequence of the binomial series 
\begin{equation}\label{eq:binomial} 
	\frac{1}{(1-z)^\alpha} = \sum_{n=0}^\infty c_\alpha(n) z^n, \qquad c_\alpha(n) := \binom{n+\alpha-1}{n}. 
\end{equation}
Burbea~\cite{Burbea87} proved the following extension of Carleman's inequality (which is the case $p=1$). If $p = 1/k$ for some integer positive $k$, then
\begin{equation}\label{eq:burbea} 
	\left(\sum_{n=0}^\infty \frac{|a_n|^2}{c_{2/p}(n)}\right)^\frac{1}{2} \leq \|f\|_{H^p}. 
\end{equation}
This paper is inspired by the following. 
\begin{conj}\label{conj:burbea} 
	The inequality \eqref{eq:burbea} holds for every $0<p\leq2$. 
\end{conj}
Our interest in Conjecture~\ref{conj:burbea} arose from a number theoretic application (see \cite{BBSSZ,BHS}), which in turn rests on the recognition of Bayart \cite{Ba} and later of Helson \cite{He} that contractive inequalities like those above may ``lift'' multiplicatively to yield interesting inequalities for Hardy spaces on the infinite-dimensional torus. As an example, we mention that Helson showed that \eqref{eq:Carl} implies a multiplicative counterpart that takes the form
\[ \left(\sum_{n=1}^{N} \frac{|b_n|^2}{d(n)}\right)^{1/2} \le \lim_{T\to \infty}\frac{1}{T}\int_{0}^T
\left| \sum_{n=1}^{N} b_n n^{-it}\right| dt, \]
where $d(n)$ is the number of divisors of $n$ and $N$ is an arbitrary positive integer.

When $p=2$, the inequality \eqref{eq:burbea} is simply an identity, since $c_1(n)=1$. That \eqref{eq:burbea} holds with some constant $C_p\geq1$ on the right hand side for every $0<p\leq2$ goes back to Hardy and Littlewood \cite{HL32}. 

The key to understanding \eqref{eq:burbea} is to observe that the left-hand side is in fact equal to the norm $\|f\|_{A^2_{2/p}}$. Here $A^p_\alpha$ is the Bergman space of the unit disc, defined for $0<p<\infty$ and $\alpha>1$ as the closure of the set of analytic polynomials with respect to the (quasi)-norm 
\begin{equation}\label{eq:Bergmannorm} 
	\|f\|_{A^p_\alpha} := \left(\int_{\mathbb{D}}|f(z)|^p\,(\alpha-1)(1-|z|^2)^\alpha \,d\mu(z)\right)^\frac{1}{p}. 
\end{equation}
In \eqref{eq:Bergmannorm} and the remainder of this paper, $\mu$ denotes the M\"obius invariant measure of the unit disc, defined for $z=x+iy$ by
\[d\mu(z) := \frac{1}{(1-|z|^2)^2}\frac{dxdy}{\pi}.\]
In light of \eqref{eq:Bergmannorm}, we note that the left hand side of \eqref{eq:burbea} becomes larger if we factor out the inner part of $f$. Setting $\alpha=2/p$ and replacing $f$ with $f^\alpha$ when $f$ is an outer function, we get the following equivalent inequality 
\begin{equation}\label{eq:reform} 
	\|f\|_{A^{2\alpha}_\alpha} = \left(\int_{\mathbb{D}}|f(z)|^{2\alpha}\,(\alpha-1)(1-|z|^2)^\alpha\, d\mu(z)\right)^\frac{1}{2\alpha}\leq\|f\|_{H^2}. 
\end{equation}
Note that Burbea's result is equivalent to the statement that \eqref{eq:reform} holds when $\alpha=k$ is a positive integer, and his result is indeed proved in this formulation. This is also the approach used by Carleman. 

We provide four conjectures and several questions. We were initially inspired by Conjecture~\ref{conj:burbea} and its symmetric companion (Conjecture~\ref{conj:dual}), both of which have been considered by others before (see e.g.~\cite[Sec.~2.5]{Pavlovic}). However, we have found our new Conjectures~\ref{conj:measure} and \ref{conj:riesz} to be more interesting. As will become clear, Conjecture~\ref{conj:measure} implies Conjecture~\ref{conj:burbea}, while the combination of Conjecture~\ref{conj:burbea} and Conjecture~\ref{conj:riesz} implies Conjecture~\ref{conj:dual}. 

This paper is organized into four further sections. In the next section, we will revisit Burbea's proof and demonstrate how it follows from a log-convexity result about the norms of the Bergman spaces $A^2_\alpha$. We will investigate the formulation \eqref{eq:reform} using calculus and duality arguments, and try to illuminate the main difficulties. In Section~\ref{sec:measure} we will discuss one possible line of attack by formulating a conjecture about a weak--type estimate for the M\"obius invariant measure $\mu$. The symmetric companion to Conjecture~\ref{conj:burbea} can be found in Section~\ref{sec:dual}, where we also formulate Conjecture~\ref{conj:riesz} regarding contractivity of the Riesz projection. In the final section, we provide numerical evidence for our conjectures and questions.

\section{Related inequalities and main difficulties} \label{sec:related} 
It is well-known that $H^p$ is the limit of $A^p_\alpha$ when $\alpha\to 1^+$ in the sense that if $f$ is in $H^p$, then $f $ is in $A^p_\alpha$ for every $\alpha>1$ and
\[\lim_{\alpha\to1^+} \|f\|_{A^p_\alpha} = \|f\|_{H^p}.\]
We therefore adopt the convention $A^p_1 = H^p$. We can rephrase \eqref{eq:reform} as the apparently more general inequality
\begin{equation}\label{eq:general} 
	\|f\|_{A^{p\alpha}_{\alpha}} \leq C_\alpha^{2/p} \|f\|_{H^p},
\end{equation}
with $C_\alpha=1$; this is achieved by observing again that we may replace $f$ by $f^{p/2}$ when $f$ is an outer function and that the left-hand side becomes larger if we factor out the inner part of $f$. 

Let $\alpha>1$ and $0<p<\infty$ be fixed. We will verify that \eqref{eq:general} is best possible for both contractivity and boundedness, and consider therefore the inequality 
\begin{equation}\label{eq:abineq} 
	\|f\|_{A^{p\alpha}_\beta} \leq C \|f\|_{H^p}. 
\end{equation}
Suppose that some choice of $\beta$ gives that \eqref{eq:abineq} holds with $C=1$. Setting $f(z)=1+\varepsilon z$ and letting $\varepsilon\to 0^+$, we can by a straightforward computation show that $\beta\geq\alpha$ is a necessary condition. Moreover, if \eqref{eq:abineq} holds for some $C\geq1$, then setting $f(z)=(1-rz)^{-2/p}$ and letting $r \to 1^-$, we get that $\beta\geq\alpha$ is necessary for mere boundedness as well.

The fact that \eqref{eq:general} holds with some constant $C_\alpha\geq1$ goes back to Hardy and Littlewood \cite[Thm.~31]{HL32}. We refer to \cite[Sec.~4.1]{Pavlovic} for a modern treatment of \eqref{eq:general} and related results, with a proof relying on Marcinkiewicz interpolation and Bourgain's decomposition lemma \cite{Bourgain48}. It must be stressed that these proofs do not give the desired constant $C_\alpha=1$ for any $\alpha>1$. In both cases, the problem seems to be that something is lost when decomposing an analytic function. Note that the case $\alpha=1$ in \eqref{eq:general} is simply the identity $\|f\|_{A^p_1} = \|f\|_{H^p}$. Carleman's inequality states that $C_2=1$.
\begin{thm}[Carleman--Burbea \cite{Burbea87,Vukotic03}] \label{thm:carlemanburbea} 
	The inequality \eqref{eq:general} holds with $C_\alpha=1$ when $\alpha=k$ is a positive integer. 
\end{thm}

As far as we are aware, it is not known whether $C_\alpha=1$ holds for any other $\alpha$. See also \cite[Sec.~2.5]{Pavlovic} for some discussion of this, and note that our Conjecture~\ref{conj:burbea} is equivalent to \cite[Prob.~2.1]{Pavlovic}. To prepare for a general version of Theorem~\ref{thm:carlemanburbea}, we begin with the following log-convexity result regarding the Bergman space norms.
\begin{lem}\label{lem:logconvex} 
	For $1\leq j \leq k$ let $f_j \in A^2_{\alpha_j}$ with $\alpha_j\geq1$. Set $f: = f_1 f_2 \cdots f_k$ and $\alpha: = \alpha_1+\alpha_2+\cdots + \alpha_k$. Then 
	\begin{equation}\label{eq:logconvex} 
		\|f\|_{A^{2}_\alpha} \leq \|f_1\|_{A^2_{\alpha_1}}\|f_2\|_{A^2_{\alpha_2}}\cdots \|f_k\|_{A^2_{\alpha_k}}. 
	\end{equation}
	Equality in \eqref{eq:logconvex} occurs if and only if there are constants $\lambda_j \in \mathbb{C}$ and $w \in \mathbb{D}$ such that $f_j(z) = \lambda_j(1-wz)^{-\alpha_j}$ for every $1 \leq j \leq k$. 
\end{lem}
\begin{proof}
	We will rely on the following formula,
	\begin{equation}\label{eq:sumprod} 
		c_{\alpha_1+\alpha_2+\cdots+\alpha_k}(n) = \sum_{n_1+n_2+\cdots + n_k = n} c_{\alpha_1}(n_1)c_{\alpha_2}(n_2)\cdots c_{\alpha_k}(n_k).
	\end{equation}
	Note that \eqref{eq:sumprod} follows at once from \eqref{eq:binomial}. It is sufficient to prove \eqref{eq:logconvex} only in the case $k=2$. Thus, let $f=gh$ where $g(z) = \sum_{n\geq0} a_n z^n$ and $h(z) = \sum_{n\geq0} b_n z^n$. We expand the left hand side of \eqref{eq:logconvex} at the level of coefficients, and use the Cauchy--Schwarz inequality with \eqref{eq:sumprod} to get 
	\begin{align*}
		\|f\|_{A^2_\alpha} &= \left(\sum_{n=0}^\infty \frac{1}{c_\alpha(n)}\left|\sum_{n_1+n_2=n}a_{n_1} b_{n_2} \right|^2\right)^\frac{1}{2} \\
		&\leq \left(\sum_{n=0}^\infty \sum_{n_1+n_2=n} \frac{|a_{n_1}|^2}{c_{\alpha_1}(n_1)} \frac{|b_{n_2}|^2}{c_{\alpha_2}(n_2)}\right)^\frac{1}{2} = \|g\|_{A^2_{\alpha_1}}\|h\|_{A^2_{\alpha_2}}, 
	\end{align*}
	where we used that $\alpha=\alpha_1+\alpha_2$. The second statement is now obvious. 
\end{proof}

The following result follows at once from Lemma~\ref{lem:logconvex} and contains Theorem~\ref{thm:carlemanburbea} as a consequence of the tautology $A^p_1 = H^p$.
\begin{cor}\label{cor:burbeaext} 
	Suppose that \eqref{eq:general} holds for some pair of parameters $0<p<\infty$ and $1\leq \alpha < \infty$. Then $\| f\|_{A^{kp\alpha}_{k\alpha}}\le C_{\alpha}^{2/p} \| f\|_{H^p}$ holds for every $0<p<\infty$ and every positive integer $k$. In particular, starting from $A^p_1 = H^p$ we get
	\[\|f\|_{A^{kp}_k} \leq \|f\|_{H^p}\]
	for every positive integer $k$. 
\end{cor}

It is also natural to ask if complex interpolation can be used to extend Theorem~\ref{thm:carlemanburbea} to non-integer values of $\alpha$. However, when we interpolate between $H^p$ spaces, a constant of interpolation appears (see \cite[pp.~18--19]{BL76} and \cite{SS48}). This constant is a direct consequence of the fact that we work with analytic functions, and do not appear when interpolating between the usual $L^p$ spaces. This seems again to be related to the decompositions needed for the interpolation machinery to work. 

For two pairs of \emph{compatible} Banach spaces $(X_0,X_1)$ and $(Y_0,Y_1)$ and an operator $T$ which is contractive from $X_0$ to $Y_0$ and $X_1$ to $Y_1$, we know that $T$ defines a contraction from the interpolation spaces $X_\theta = [X_0,X_1]_\theta$ to the interpolation spaces $Y_\theta = [Y_0,Y_1]_\theta$, for $0<\theta<1$. The statement about Hardy spaces made above is that $[H^p,H^q]_\theta \cong H^r$ for some suitable $r$ between $p$ and $q$, but the norms are equivalent and not equal. This means that it is impossible to obtain contractive results in this way. 

However, by fixing $p=2$ so that $[H^2,H^2]_\theta = H^2$ we use the formulation \eqref{eq:reform}, and interpolate between the cases $\alpha=2,3,\ldots$ to retain contractive estimates. Note that the case $\alpha=1$ has to be excluded since the norm is supported on $\mathbb{T}$ and not in $\mathbb{D}$ --- or more technically, the Dirac measure is not absolutely continuous with respect to the Lebesgue measure (see \cite[Chap.~5]{BL76}). 

This interpolation procedure can also be carried out directly using the three lines lemma as in the proof of the Riesz--Thorin theorem, but we give a shorter (and essentially equivalent) formulation using interpolation spaces. 
 Let $[\alpha]$ denote the integer part of the positive real number $\alpha$ and let $\{\alpha\}$ denote the fractional part of $\alpha$, so that $\alpha=[\alpha]+\{\alpha\}$.
\begin{lem}\label{lem:interpolation} 
	Let $\alpha\geq2$ and suppose that $f$ is  in $H^2$. Then
	\[\|f\|_{A^{2\alpha}_\alpha} \leq \left(\frac{\alpha-1}{\left([\alpha]-1\right)^{1-\{\alpha\}}\left([\alpha] \right)^{\{\alpha\}}}\right)^\frac{1}{2\alpha} \|f\|_{H^2}.\]
\end{lem}
\begin{proof}
	In view of \eqref{eq:reform}, we consider $T$ to be the operator defined by
	\[(a_n)_{n\geq0} \qquad\mapsto\qquad \sum_{n=0}^\infty a_n z^n,\]
	which maps $\ell^2$ to $L^p(\mathbb{D},dA_\alpha)$, where
	\[dA_\alpha(z) = (\alpha-1)(1-|z|^2)^{\alpha}\,d\mu(z).\]
	We will interpolate between the cases $\alpha=k$ and $\alpha=k+1$, for $k=2,3,\ldots$, where we know that $T$ is contractive. Let therefore $k<\alpha<k+1$ and set $\theta=\{\alpha\}$. Obviously $[\ell^2,\ell^2]_\theta = \ell^2$. We get from \cite[Sec.~5.5]{BL76} that 
	\[\left[L^{2k}(\mathbb{D},dA_k),\,L^{2(k+1)}(\mathbb{D},dA_{k+1})\right]_\theta = L^{2(k+\theta)}(\mathbb{D},d\widetilde{A}_\theta),\]
	with equal norms, where
	\[d\widetilde{A}_\theta(z) = \left(k-1\right)^{1-\theta}\left(k+1-1\right)^{\theta} (1-|z|^2)^{k+\theta} \,d\mu(z).\]
	The proof is completed by recalling that $k=[\alpha]$ and $\theta=\{\alpha\}$. 
\end{proof}
It is plain that 
\begin{equation}\label{eq:limalpha} 
	\lim_{\alpha\to \infty} \frac{\alpha-1}{\left([\alpha]-1\right)^{1-\{\alpha\}}\left([\alpha] \right)^{\{\alpha\}}} = 1. 
\end{equation}
We offer only the following explicit value, formulated with respect to Conjecture~\ref{conj:burbea}.
\begin{thm}\label{thm:smallconstant} 
	Let $0<p<1$ and set $C=\left(2/(e\log{2})\right)^{1/2}=1.030279\ldots$. Then
	\[\left(\sum_{n=0}^\infty \frac{|a_n|^2}{c_{2/p}(n)}\right)^\frac{1}{2} \leq C\|f\|_{H^p}.\]
\end{thm}
\begin{proof}
	We set $\alpha=2/p$ and reformulate Lemma~\ref{lem:interpolation} as follows. If $0<p<1$ and $f(z) = \sum_{n\geq0} a_n z^n$ is in $H^p$, then
	\[\left(\sum_{n=0}^\infty \frac{|a_n|^2}{c_\alpha(n)}\right)^\frac{1}{2} \leq \left(\frac{\alpha-1}{\left([\alpha]-1\right)^{1-\{\alpha\}}\left([\alpha] \right)^{\{\alpha\}}}\right)^\frac{1}{2} \|f\|_{H^p}.\]
	Set $[\alpha]=k$ and note that
	\[ \frac{\alpha-1}{\left([\alpha]-1\right)^{1-\{\alpha\}} \left([\alpha]\right)^{\{\alpha\}}} = \frac{k^k}{(k-1)^{k+1}}\cdot (\alpha-1)\left(\frac{k}{k-1}\right)^{-\alpha}. \]
	The maximum of
	\[\alpha \mapsto (\alpha-1)\left(\frac{k}{k-1}\right)^{-\alpha}\]
	on the interval $[k,k+1]$ occurs at the point
	\[\alpha = 1 + 1/\log\left(\frac{k}{k-1}\right) =: 1 + 1/\beta(k).\]
	Hence we get that 
	\begin{align*}
		\frac{\alpha-1}{\left([\alpha]-1\right)^{1-\{\alpha\}} \left([\alpha]\right)^{\{ \alpha\}}} &\leq \frac{1}{\beta(k)(k-1)^{k-1/\beta(k)} k^{1+1/\beta(k)-k}} \\
		&=\exp\left(k\beta(k)-\log\left(k\beta(k)\right)-1\right). 
	\end{align*}
	We easily check that $x \mapsto x-\log(x)-1$ is increasing and that $k \mapsto k\beta(k)$ is decreasing, and conclude that for every $\alpha\geq2$ we have
	\[\frac{\alpha-1}{\left([\alpha]-1\right)^{1-\{\alpha\}} \left([\alpha]\right)^{\{\alpha\}}} \leq \exp\left(2\beta(2)-\log\left(2\beta(2)\right)-1\right) = \frac{2}{e\log{2}}. \qedhere\]
\end{proof}

The final observation of this section is that \eqref{eq:general} seems to be weaker as $\alpha$ increases. To make this statement explicit, let us now fix an outer function $f$  in $H^2$, and assume that $\|f\|_{H^2}=1$. We define 
\begin{equation}\label{eq:Uf} 
	U_f(\alpha) := \|f\|_{A^{2\alpha}_\alpha}^{2\alpha} = \int_{\mathbb{D}}|f(z)|^{2\alpha}\,(\alpha-1)(1-|z|^2)^\alpha \,d\mu(z).
\end{equation}
Replacing $\alpha$ by a complex number $\gamma = \alpha+i\beta$, we use the fact that $f$ is outer, and hence non-vanishing, to conclude that
\[U_f(\gamma) = \mre \int_{\mathbb{D}}|f(z)|^{2\gamma}\,(\gamma-1)(1-|z|^2)^\gamma \,d\mu(z)\]
is harmonic. Clearly $U_f(1)=1$, and it would be enough to prove that $U_f(\alpha)\leq1$ for $\alpha\geq1$ to conclude. We are therefore led to the the following. 
\begin{que}\label{que:decreasing} 
	Suppose that $f$ is an outer function with $\|f\|_{H^2}=1$. Is it true that
	\[\frac{
	\partial}{
	\partial \alpha} U_f(\alpha)\le 0\]
	for all $\alpha>1$? 
\end{que}
Let us collect some computations in support of a positive answer to this question. We see that 
\begin{align*}
	U'_f(\alpha) := \frac{
	\partial}{
	\partial \alpha} U_f(\alpha) = & \int_{\mathbb{D}} |f(z)|^{2\alpha}(1-|z|^2)^{\alpha} d\mu(z) \\
	& - (\alpha-1) \int_{\mathbb{D}} |f(z)|^{2\alpha}(1-|z|^2)^{\alpha} \log\frac{1}{|f(z)|^{2}(1-|z|^2)} d\mu(z). 
\end{align*}
Note that if $f$ is not a normalized reproducing kernel, then by the Cauchy--Schwarz inequality
\[|f(z)|^{2}(1-|z|^2) < \|f\|_{H^2}^2=1\]
for every point $z \in \mathbb{D}$. This means that $U_f$ is ``eventually'' decreasing because the logarithmic factor is bounded below. Note that by Lemma~\ref{lem:logconvex} and the fact that $f$ is non-vanishing we get that
\[U_f(\alpha+\beta) \leq U_f(\alpha)U_f(\beta)\]
for every $\alpha,\beta\geq1$. In particular, $U_f(\alpha+1)\leq U_f(\alpha)$, so that $ U'_f(\alpha)\le 0$ must occur for some $\alpha$ in every interval of length $1$.


\section{A weak type inequality} \label{sec:measure} 
Let again $f$ be a nontrivial function in $H^2$ with $\| f\|_{H^2}=1$. For $0<\lambda<1$, we wish to estimate $\mu(E_f(\lambda))$, where
\[ E_f(\lambda):=\left\{z\in \mathbb{D}: \, |f(z)|^2(1-|z|^2) > \lambda \right\}. \]
Our interest in $\mu(E_f(\lambda))$ is motivated by the observation that \eqref{eq:Uf} may rewritten in distributional form as 
\begin{equation}\label{eq:Uf2} 
	U_f(\alpha)=\alpha(\alpha-1) \int_0^{1} \lambda^{\alpha-1} \mu(E_f (\lambda) )\,d\lambda. 
\end{equation}
We will prove the following theorem. 
\begin{thm}\label{thm:meas} 
	There exists a universal constant $C$ such that 
	\begin{equation}\label{eq:bound} 
		\mu(E_f(\lambda)) \le C(1/\lambda-1) 
	\end{equation}
	for every function $f$ in $H^2$ with $\|f\|_{H^2}=1$ and $0<\lambda<1$. 
\end{thm}
Plugging \eqref{eq:bound} into \eqref{eq:Uf2}, we get
\[ U_f(\alpha)\le C. \]
Thus if we could prove the following conjecture, then we would also have proven Conjecture~\ref{conj:burbea}.
\begin{conj}\label{conj:measure} 
	Suppose that $f \in H^2$. Then
	\[\mu\left(\left\{z\in \mathbb{D}\,: \, |f(z)|^2(1-|z|^2) > \lambda \|f\|_{H^2}^2\right\}\right) \leq \frac{1}{\lambda}-1.\]
\end{conj}
We will make some additional comments on this conjecture below and proceed now to prove Theorem~\ref{thm:meas}. 
\begin{proof}
	[Proof of Theorem~\ref{thm:meas}] By M\"{o}bius invariance, we may, without loss of generality, assume that
	\[z\mapsto |f(z)|^2(1-|z|^2)\]
	attains its maximum at $0$. Then for every $0\le \theta<2\pi$ and $\lambda$ less than $|f(0)|$,
	\[ r^*(\theta):=\max \left\{ r\,:\, |f(re^{i\theta})|^2(1-r^2)=\lambda \right\} \]
	exists. It follows that
	\[ \mu(E_f(\lambda))\le \int_0^{2\pi} \int_{0}^{r^*(\theta)} \frac{2rdr}{(1-r^2)^2}\,\frac{d\theta}{2\pi} \]
	and hence 
	\begin{equation}\label{eq:radial} 
		\mu(E_f(\lambda))\le \frac{1}{\lambda} \int_{0}^{2\pi} \max\left\{|f(re^{i\theta})|^2 \,:\,|f(re^{i\theta})|^2(1-r^2)=\lambda\right \}\, \frac{d\theta}{2\pi} - 1.
	\end{equation}
	We introduce the radial maximal function $f^*(\theta):=\sup_{0<r<1} |f(re^{i\theta})|$ and see that the latter inequality implies that 
	\begin{equation}\label{eq:weak} 
		\mu(E_f(\lambda))\le \frac{1}{\lambda} \int_0^{2\pi} f^*(\theta)^2 \,\frac{d\theta}{2\pi} -1 \le \frac{C}{\lambda}\|f\|_{H^2}^2-1=\frac{C}{\lambda}-1,
	\end{equation}
	where we in the second step used a classical theorem of Hardy and Littlewood.
	
	We can improve \eqref{eq:weak} by using \eqref{eq:general}, or in other words the fact that there is an absolute constant $M$ (independent of $f$) such that
	\[k(k-1)\int_0^1 \lambda^{k-1} \mu(E_f(\lambda)) \,d\lambda \le M \]
	whenever, say, $k\ge 2$. (In view of \eqref{eq:limalpha}, $M$ can be chosen arbitrarily close to $1$ if we choose a sufficiently large lower bound for $k$.) Taking into account the monotonicity of $\mu(E_f(\lambda))$, we see that
	\[\mu(E_f(1-\varepsilon)) \le \frac{M(1-\varepsilon)^{-k}}{k-1} = M (1-\varepsilon)^{-1} \frac{(1-\varepsilon)^{-(k-1)}}{k-1} . \]
	The maximum of $x\mapsto c^x/x$ occurs when $x=1/|\log(1-\varepsilon)|$ so that
	\[\mu(E_f(1-\varepsilon))\le Me (1-\varepsilon)^{-1} |\log(1-\varepsilon)|=Me(\varepsilon+O(\varepsilon^2))\]
	when $\varepsilon \to 0^+$. The estimate in \eqref{eq:bound} can therefore be improved to
	\[\mu(E_f(\lambda)) \le C'\left(\frac{1}{\lambda}-1\right)\]
	for a universal constant $C'$. 
\end{proof}
Clearly, the transition from \eqref{eq:radial} to \eqref{eq:weak} is not optimal. In fact, we could hope for an affirmative answer to the following question, which would establish Conjecture~\ref{conj:measure}.
\begin{que} \label{que:increase}
	Is it true that
	\begin{equation}\label{eq:increase} 
		\int_{0}^{2\pi} \max\left\{|f(re^{i\theta})|^2 \,:\,|f(re^{i\theta})|^2(1-r^2)=\lambda\right\} \,\frac{d\theta}{2\pi} \le \| f\|_{H^2}^2 
	\end{equation}
	holds whenever $0<\lambda<|f(0)|$?
\end{que} 
It is clear that \eqref{eq:increase} holds for $\lambda$ close to 1; this is obvious when $f$ is a reproducing kernel and otherwise we have $|f(0)|<\| f\|_{H^2}$. Since we may assume that $f$ is an outer function, it is also clear that the integral in \eqref{eq:radial} tends to $\| f\|_{H^2}^2$ when $\lambda \to 0$, in view of Lebesgue's dominated convergence theorem and the domination by $f^*$. Hence if it is true that the integral increases with $1/\lambda$, then \eqref{eq:increase} would follow. 

Alternatively, we may approach Conjecture~\ref{conj:measure} by starting from the formula
\[ \mu(E_f(\lambda))=\int_{E_f(\lambda)} d\mu(z) = \int_{E_f(\lambda)} \Delta \log \left(\frac 1{1-|z|^2}\right)\frac{dxdy}{4\pi}. \]
Using the second Green identity, we obtain 
\begin{align*}
	\int_{E_f(\lambda)} \Delta \log \left(\frac 1{1-|z|^2}\right) \frac{dxdy}{4\pi} &=\int_{
	\partial E_f(\lambda)} \frac {
	\partial}{
	\partial n} \log \left(\frac 1{1-|z|^2}\right) \frac{ds}{4\pi} \\&= \int_{
	\partial E_f(\lambda)} \frac 1{1-|z|^2} \frac {
	\partial}{
	\partial n} |z|^2 \frac{ds}{4\pi},
\end{align*}
where $\frac{
\partial}{
\partial n}$ is differentiation in the outward normal direction and $ds$ is the arclength measure over the curve $\gamma_\lambda= 
\partial E_f(\lambda)$. On the curve $\gamma_\lambda$, we have that
\[\frac 1{1-|z|^2} = \frac {|f(z)|^2}{\lambda}\]
where we retain the normalization $\|f\|_{H^2}^2=1$. Thus 
\begin{align*}
	\int_{\gamma_\lambda} \frac 1{1-|z|^2} \frac {
	\partial}{
	\partial n} |z|^2\,\frac{ds}{4\pi} &= \int_{\gamma_\lambda} \frac {|z|^2}{1-|z|^2} \frac{
	\partial}{
	\partial n} \log |z|^2\, \frac{ds}{4\pi} \\
	&= \int_{\gamma_\lambda} \left(\frac{|f(z)|^2}{\lambda}-1\right) \frac{
	\partial}{
	\partial n} \log |z|^2\, \frac{ds}{4\pi}. 
\end{align*}
Finally, using again the second Green identity, we get 
\begin{align*}
	\int_{\gamma_\lambda} \left(\frac{|f(z)|^2}{\lambda}-1\right) \frac{ 
	\partial}{ 
	\partial n} \log |z|^2\, \frac{ds}{4\pi}= & 
	-\frac1{\lambda}\int_{\gamma_\lambda} \log |z|^2\frac{ 
	\partial}{ 
	\partial n}|f(z)|^2 \, \frac{ds}{4\pi} \\ & + \int_{E_f(\lambda)} \frac 1{\lambda}\log |z|^2\Delta |f(z)|^2\frac{dxdy}{4\pi} +\frac{|f(0)|^2}{\lambda} -1. 
\end{align*}
If we use the Littlewood--Paley formula for the $H^2$ norm, namely
\[ \frac 1{2\pi} \int_0^{2\pi} |f(e^{i\theta})|^2 \, d\theta = |f(0)|^2 + \frac 1{\pi}\int_{\mathbb D} |f'(z)|^2\log \frac 1{|z|^2}\, dxdy, \]
then we finally find that
\[\mu(E_f(\lambda)) = \frac 1\lambda-1- \frac {1}{\lambda}\int_{\mathbb{D}\setminus E_f(\lambda)}|f'(z)|^2\log \frac1{|z|^2}\, \frac{dxdy}{\pi} + \frac 1{\lambda}\int_{\gamma_\lambda} \log |z|^2\frac{
\partial}{ 
\partial n}|f(z)|^2 \, \frac{ds}{4\pi}. \]
If the last integral is negative, then we have proved the desired inequality and Conjecture~\ref{conj:measure} is verified. 


\section{Contractive symmetry and Riesz projection} \label{sec:dual} 
We will now investigate a symmetric companion to \eqref{eq:burbea} in the range $2\leq q < \infty$. To be more precise, we are interested in contractive inequalities of the type 
\begin{equation}\label{eq:genupper} 
	\|f\|_{H^q} \leq \left(\sum_{n=0}^\infty |a_n|^2 c_\beta(n)\right)^\frac{1}{2} =: \|f\|_{D_\beta}. 
\end{equation}
By choosing $f(z) = 1+\varepsilon z$ for some small $\varepsilon>0$, it is easy to verify that $\beta \geq q/2$ is necessary for \eqref{eq:genupper} to be contractive.
\begin{conj}\label{conj:dual} 
	The inequality 
	\begin{equation}\label{eq:dual} 
		\|f\|_{H^q} \leq \left(\sum_{n=0}^\infty |a_n|^2 c_{q/2}(n)\right)^\frac{1}{2} 
	\end{equation}
	holds for every $2\leq q < \infty$. 
\end{conj}

Clearly, \eqref{eq:dual} reduces to an identity when $q=2$, so we get that \eqref{eq:dual} holds for all even integers $q$ from the following result, which is an analogue of Corollary~\ref{cor:burbeaext}.
\begin{lem}\label{lem:dualext} 
	If \eqref{eq:genupper} holds for some pair of parameters $2 \leq q < \infty$ and $1\leq \beta < \infty$, then it also holds for the pairs $kq$ and $k\beta$ for every positive integer $k$. 
\end{lem}
\begin{proof}
	We will rely on two preliminary results. The first is \eqref{eq:sumprod} and the second is
	\[c_\beta(n_1+n_2+\cdots+n_k) \leq c_\beta(n_1)c_\beta(n_2)\cdots c_\beta(n_k),\]
	which can be deduced inductively from the case $k=2$. Set $n=n_1+n_2$, so 
	\begin{align*}
		\binom{n + \beta-1}{n} = \prod_{j=1}^n \frac{j+\beta-1}{j} &= \binom{n_1+\beta-1}{n_1} \prod_{j=1}^{n_2} \frac{n_1+j+\beta-1}{n_1+j} \\
		&\leq \binom{n_1+\beta-1}{n_1}\binom{n_2+\beta-1}{n_2}, 
	\end{align*}
	where we used that $\beta-1\geq0$. Let $f(z) = \sum_{n\geq0} a_n z^n$. Then by the assumption that $\|g\|_{H^q}\leq \|g\|_{D_\beta}$, the Cauchy--Schwarz inequality and the preliminary results, we get that 
	\begin{align*}
		\|f\|_{H^{kq}}^k &= \|f^k\|_{H^q} \leq \|f^k\|_{D_\beta}=\left(\sum_{n=0}^\infty c_\beta(n)\left|\sum_{n_1+\cdots n_k = n} a_{n_1} \cdots a_{n_k}\right|^2\right)^\frac{1}{2} \\
		&\leq \left(\sum_{n=0}^\infty c_\beta(n)\left[\,\sum_{n_1+\cdots+n_k=n} c_\beta(n_1)\cdots c_\beta(n_k)\right]\right. \\
		&\qquad\qquad\qquad\qquad\qquad\, \times\, \left.\left[\,\sum_{n_1+\cdots+n_k=n} \frac{|a_{n_1}|^2}{c_\beta(n_1)}\cdots \frac{|a_{n_k}|^2}{c_\beta(n_k)}\right]\right)^\frac{1}{2} \\
		&=\left(\sum_{n=0}^\infty c_\beta(n) c_{k\beta}(n)\sum_{n_1+\cdots+n_k=n}\frac{|a_{n_1}|^2}{c_\beta(n_1)}\cdots \frac{|a_{n_k}|^2}{c_\beta(n_k)}\right)^\frac{1}{2} \\
		&\leq \left(\sum_{n=0}^\infty \sum_{n_1+\cdots + n_k = n} c_{k\beta}(n_1) |a_{n_1}|^2 \cdots c_{k\beta}(n_k) |a_{n_k}|^2 \right)^\frac{1}{2} = \|f\|_{D_{k\beta}}^k, 
	\end{align*}
	and we are done. 
\end{proof}

Let us now turn to a direct connection between the inequalities \eqref{eq:burbea} and \eqref{eq:dual}. We will work here in the pairing of $L^2(\mathbb{T})$, and view therefore the Hardy space $H^p$ as a closed subspace of $L^p(\mathbb{T})$. Through the Hahn--Banach theorem, this gives that the dual space $(H^p)^\ast$ is equal (with identical norm) to the quotient space $L^q/(L^q\ominus H^q)$, where $1/p+1/q=1$. Finally, when $1<p<\infty$, we have $L^q/(L^q\ominus H^q)\cong H^q$, but the norms are not equal. The best general bounds we can give are
\[\|g\|_{L^q/(L^q\ominus H^q)} \leq \|g\|_{H^q} \leq \frac{1}{\sin(\pi/q)}\|g\|_{L^q/(L^q\ominus H^q)},\]
the upper (sharp) bound is from \cite{HV00}. Let $P$ denote the Riesz projection from $L^2(\mathbb{T})$ to $H^2$. It is well-known that $P$ extends to a bounded operator from $L^q(\mathbb{T})$ to $H^q$, and the result from \cite{HV00} gives that the norm is $(\sin(\pi q))^{-1}$.

We now compare the dual of \eqref{eq:burbea} (taken here only for $1\leq p \leq 2$) with \eqref{eq:dual} (taken here only for $2\leq p \leq 4$). Clearly, the dual space of $A^2_\alpha$ is isometrically isomorphic to $D_\alpha$ in this pairing, so we get that 
\begin{align*}
	\|f\|_{\left(H^p\right)^\ast} &\leq \|f\|_{D_{2/p}}, \qquad 1\leq p \leq 2, \\
	\|f\|_{H^q} &\leq \|f\|_{D_{q/2}}, \qquad 2 \leq q\leq 4. 
\end{align*}

In particular, this means that \eqref{eq:genupper} holds with $\beta=2(1-1/q)$, but it is not contractive unless $q=2$ since $2(1-1/q)<q/2$. Nevertheless, setting $2/p=q/2$, we conjecture a contractive relationship between $H^q$ and the Riesz projection of $L^r(\mathbb{T})$ where
\[q = \frac{4}{p} = 4\left(1-\frac{1}{r}\right).\]
A result from \cite{MZ11} states that the Riesz projection $P\colon L^\infty(\mathbb{T}) \to H^4$ is contractive and that $q=4$ is sharp. Hence we arrive at the following. 
\begin{conj}\label{conj:riesz} 
	The Riesz projection $P$ is a contraction from $L^r(\mathbb{T})$ to $H^q$ with $q = 4(1-1/r)$ for $1<r \leq \infty$. If $f$ is in $L^1(\mathbb{T})$, then
	\begin{equation} \label{eq:rieszL1}
		\exp\left(\int_0^{2\pi}\log|Pf(e^{i\theta})|\,\frac{d\theta}{2\pi}\right) \leq \|f\|_{L^1}.
	\end{equation}
\end{conj}

Here follow several remarks pertaining to Conjecture~\ref{conj:riesz}. Let us begin by making note of two observations regarding \eqref{eq:rieszL1}. A classical theorem (see~\cite[Chap.~VII.2]{Zygmund}) states that there is an absolute constant $C\geq1$ such that for every $0<q<1$,
\begin{equation} \label{eq:zygmund}
	\|Pf\|_{H^q} \leq \frac{C}{1-q}\|f\|_{L^1}.
\end{equation}
As far as we know, the best constant $C_q$ in \eqref{eq:zygmund} is not known. Suppose now that $g$ is in $H^q$ for some $q>0$. Then
\begin{equation} \label{eq:lim0}
	\lim_{q \to 0^-} \|g\|_{H^q} = \exp\left(\int_0^{2\pi}\log|g(e^{i\theta})|\,\frac{d\theta}{2\pi}\right),
\end{equation}
which means that \eqref{eq:rieszL1} can be interpreted as saying that \eqref{eq:zygmund} is contractive in the limit $q \to 0^-$. Moreover, from \eqref{eq:lim0} we also see that if the statement of Conjecture~\ref{conj:riesz} holds for $1<r\leq1+\delta$, for some positive $\delta$, then the statement from $r=1$ follows by taking said limit.

Since the Riesz projection is a self-adjoint operator, the above-mentioned result from \cite{MZ11} gives that the Riesz projection is contractive from $L^{4/3}(\mathbb{T})$ to $H^1$. Indeed, if $g$ is  in $L^\infty(\mathbb{T})$, then
\begin{equation} \label{eq:43}
	\left|\langle Pf, g \rangle_{L^2(\mathbb{T})}\right| = \left| \langle f, Pg \rangle_{L^2(\mathbb{T})}\right| \leq \|f\|_{L^{4/3}}\|Pg\|_{H^4} \leq \|f\|_{L^{4/3}}\|g\|_{L^\infty}.
\end{equation}
We have therefore verified that Conjecture~\ref{conj:riesz} is true for $r=4/3$. This argument is a special case of the following result, where we use the notation $r^{\ast}$ for the conjugate exponent to $r$ when $1\le r \le \infty$, i.e. $1/r+1/r^{\ast}=1$. 
\begin{lem} \label{lem:dualexp}
	Suppose that Conjecture~\ref{conj:riesz} holds for some $r$ and that $q = 4(1-1/r)>1$. Then Conjecture~\ref{conj:riesz} also holds with $r = q^\ast$ and $q = r^\ast$. In particular, Conjecture~\ref{conj:riesz} holds in the interval $2 \leq r \leq \infty$ if and only if it holds in the interval $4/3 \leq r \leq 2$.
\end{lem}

\begin{proof}
	Arguing as in \eqref{eq:43}, we get that $P \colon L^r(\mathbb{T}) \to H^q$ is contractive if and only if $P \colon L^{q^\ast}(\mathbb{T}) \to H^{r^\ast}$ is contractive. We now rewrite the relationship $q = 4(1-1/r)$ as
	\begin{equation} \label{eq:dual4}
		qr^\ast = 4.
	\end{equation}
	If \eqref{eq:dual4} holds for $q$ and $r$, then it also holds for $q = r^\ast$ and $r = q^\ast$, since $(q^{\ast})^\ast = q$.
\end{proof}

By complex interpolation between the contractive cases $P\colon L^\infty(\mathbb{T}) \to H^4$ and $P\colon L^2(\mathbb{T}) \to H^2$, it was found in \cite{MZ11} that the Riesz projection is contractive from $L^r(\mathbb{T})$ to $H^{4r/(r+2)}$ for $2 \leq r \leq \infty$. Furthermore, by arguing as in Lemma~\ref{lem:dualexp}, we find that the Riesz projection is contractive from $L^r(\mathbb{T})$ to $H^{2r/(4-r)}$ when $4/3 \leq r \leq 2$. However, note that
\begin{align*}
	\frac{4r}{r+2}&<4\left(1-\frac{1}{r}\right), \qquad 2<r<\infty, \\
	\frac{2r}{4-r}&<4\left(1-\frac{1}{r}\right), \qquad \frac{4}{3}<r<2,
\end{align*}
so again we observe that interpolation gives something weaker than our conjecture. Let us now see that $q\leq4(1-1/r)$ in Conjecture~\ref{conj:riesz} is necessary. 
\begin{thm}\label{thm:rieszneccessary} 
	Let $1<r<\infty$. If the Riesz projection $P$ is contractive from $L^r(\mathbb T)$ to $H^q$, then $q\leq 4(1-1/r)$. There is no $q>0$ such that $P$ is contractive from $L^1(\mathbb T)$ to $H^q$. 
\end{thm}
\begin{proof}
	Fix $1\le r<\infty$ and consider the function $f(z)=(1-\varepsilon z)/(1-\varepsilon \bar{z})^{1-2/r}$ for some small $\varepsilon>0$. Then
\[\|f\|_{L^r} = \left(\int_0^{2\pi} |1-\varepsilon e^{i\theta}|^2\,\frac{d\theta}{2\pi}\right)^{1/r} = 1+ \frac{1}{r} \varepsilon^2 + O(\varepsilon^4).\]
	Note that $Pf(z)=1-\varrho\varepsilon^2 - \varepsilon z$ with $\varrho = 1-2/r$. Hence 
	\begin{align*}
		\|Pf\|_{H^q} & = \left( \int_0^{2\pi} | 1-\varrho\varepsilon^2 - \varepsilon e^{i\theta}|^q\, \frac{d\theta}{2\pi} \right)^{1/q} \\
		& = ( 1-\varrho\varepsilon^2 ) \left( \int_0^{2\pi} \left| \left(1 - \frac{\varepsilon}{1-\varrho\varepsilon^2 } e^{i\theta} \right)^{q/2} \right|^2\,\frac{d\theta}{2\pi} \right)^{1/q} \\
		& = ( 1-\varrho\varepsilon^2 )\left(1+\frac{q}{4}\frac{\varepsilon^2}{(1-\varrho\varepsilon^2)^2} +O(\varepsilon^4)\right) = 1-\varrho\varepsilon^2 + \frac{q}{4}\varepsilon^2 + O(\varepsilon^4). 
	\end{align*}
	As $\varepsilon\to0$, contractivity implies that
	\[1+ \frac{1}{r} \varepsilon^2 \geq 1-\varrho\varepsilon^2 + \frac{q}{4}\varepsilon^2 \qquad \Longleftrightarrow \qquad \frac{q}{4} \leq \frac{1}{r}+\varrho = 1-\frac{1}{r},\]
	so we conclude that $q\leq 4(1-1/r)$ is necessary for contractivity. Moreover, when $r=1$, there is no $q>0$ such that $P$ is contractive from $L^1(\mathbb{T})$ to $H^q$. 
\end{proof}

Let us end this section by making precise the relationship between Conjecture~\ref{conj:burbea}, Conjecture~\ref{conj:dual}, and Conjecture~\ref{conj:riesz}.
\begin{thm}\label{thm:connect} 
	Suppose that Conjecture~\ref{conj:burbea} holds for $1 < p < 2$ and that Conjecture~\ref{conj:riesz} holds for $2 < r < \infty$. Then Conjecture~\ref{conj:dual} holds for $2 < q<\infty$. 
\end{thm}
\begin{proof}
	It is sufficient to verify that Conjecture~\ref{conj:dual} holds for $2<q<4$, by Lemma~\ref{lem:dualext}. For $2<q<4$, we use Conjecture~\ref{conj:dual} with $r=4/(4-q)$ and take the infimum over all $g \in L^r(\mathbb{T})$ such that $Pg=f$. Hence we get
	\[\|f\|_{H^q} \leq \inf_{Pg = f} \|g\|_{L^{4/(4-q)}} = \|f\|_{(H^{4/q})^\ast} \leq \|f\|_{D_{q/2}},\]
	where the final inequality is the (the dual of) Corollary~\ref{conj:burbea}. 
\end{proof}

\section{Numerical evidence} \label{sec:numerical} 

Conjecture~\ref{conj:burbea} has been tested by generating random polynomials $f$ and computing with such functions. More precisely, we fixed the degree and picked the coefficients $a_n$ of the polynomials independently with a normal distribution centered at $0$ and with variance $c_{2/p}(n)$. We then computed the $L^p(\mathbb T)$ of $f$ norm with an adaptative Gauss--Kronrod quadrature. We tested the conjecture against several thousand polynomials successfully.

As for Conjecture~\ref{conj:measure}, we have some numerical evidence that the integral in \eqref{eq:radial} is increasing with $1/\lambda$. We took several thousand normalized random polynomials as before. For each polynomial $f$ we found the maximum of $(1-|z|^2)|f(z)|^2$. We composed each polynomial with a M\"obius transformation in such a way that the new function $g$ had the property that the maximum of $(1-|z|^2)|g(z)|^2$ is attained at the origin. For fixed $\theta$ and $\lambda$, it is possible to solve numerically  the equation $|g(r e^{i\theta})|^2(1-r^2) = \lambda$ and thus we computed numerically the integral in \eqref{eq:radial} for many values of $0<\lambda<|g(0)|$, and we have checked that it actually increases with $1/\lambda$. We have also checked numerically that Question~\ref{que:decreasing} has a positive answer for $1<\alpha<2$ for many random polynomials.

Conjecture~\ref{conj:dual} has been tested successfully in a similar way as Conjecture~\ref{conj:burbea}. To test Conjecture~\ref{conj:riesz}, we considered random trigonometric polynomials of the form
\[f(z) = \sum_{n=-M}^{N} a_n z^n\]
for different values of $M,N>0$ and independent coefficients with standard normal distribution. We computed numerically the $L^r(\mathbb T)$ norm of $f$ and the $L^q(\mathbb T)$ norm of its Riesz projection
\[Pf(z)=\sum_{n=0}^{N} a_n z^n.\]
We observed that the Riesz projection was contractive for $q \ge  4(1-1/r)$. Moreover, in the limiting case $r=1$, we observed that
\[ \exp\left(\int_0^{2\pi} \log |Pf(e^{i\theta})|\, \frac{d\theta}{2\pi}\right) \le \int_0^{2\pi} |f(e^{i\theta})|\,\frac{d\theta}{2\pi}. \]

\bibliographystyle{amsplain} 
\bibliography{contractive} \end{document}